\documentclass[reqno]{amsart}

\usepackage{amsmath,amsfonts,amstext,amssymb,amscd,amsthm,mathrsfs}
\usepackage{xcolor}

\usepackage[textwidth=33em,hmarginratio=1:1,%
textheight=1.5\textwidth,paperwidth=44em,paperheight=66em]{geometry}

\makeatletter
\newcommand{\newsectionstyle}{%
  \renewcommand{\@secnumfont}{\bfseries}
 \renewcommand\section{\@startsection{section}{1}%
\z@{.7\linespacing\@plus\linespacing}{.5\linespacing}%
{\large\bfseries\centering}}
}
\let\oldsection\section%
\let\old@secnumfont\@secnumfont%
\newcommand{\originalsectionstyle}{%
  \let\@secnumfont\old@secnumfont
  \let\section\oldsection
}
\makeatother
  
\newenvironment{nouppercase}{%
  \renewcommand{\uppercasenonmath}[1]{}}{}

\newcommand{\tld}{\widetilde} 
\newcommand{\nb}{\nabla}
\newcommand{\p}{\partial}
\newcommand{\ep}{\varepsilon}

\newcommand{\R}{\mathbb R}
\newcommand{\Z}{\mathbb Z}

\DeclareMathOperator {\argmin}{argmin}
\DeclareMathOperator {\dist}{dist}

\newtheorem{theorem}{Theorem}
\newtheorem{proposition}[theorem]{Proposition}
\newtheorem{lemma}[theorem]{Lemma}
\newtheorem{corollary}[theorem]{Corollary}

\theoremstyle{definition}

\newtheorem{definition}[theorem]{Definition}

\theoremstyle{remark}
\newtheorem{remark}[theorem]{Remark}

\numberwithin{equation}{section}
\numberwithin{theorem}{section}

\usepackage[initials]{amsrefs}
\usepackage{url}
\usepackage{hyperref}

\begin{document}
 \title[A minimising movement scheme for the $p$-elastic energy of curves]{\LARGE\mdseries 
 A minimising movement scheme for the $p$-elastic energy of curves}
 \author[Blatt]{Simon Blatt}
 \address[Simon Blatt]{Departement of Mathematics, Paris Lodron Universit\"at Salzburg, Hellbrunner Strasse 34, 5020 Salzburg, Austria}
\email[Simon Blatt]{simon.blatt@sbg.ac.at}
\author[Hopper]{Christopher P. Hopper}
 \address[Christopher P. Hopper]{Departement of Mathematics, Paris Lodron Universit\"at Salzburg, Hellbrunner Strasse 34, 5020 Salzburg, Austria}
\email[Christopher P. Hopper]{christopher.hopper@sbg.ac.at}
\author[Vorderobermeier]{Nicole Vorderobermeier}
 \address[Nicole Vorderobermeier]{Departement of Mathematics, Paris Lodron Universit\"at Salzburg, Hellbrunner Strasse 34, 5020 Salzburg, Austria}
\email[Nicole Vorderobermeier]{nicole.vorderobermeier@sbg.ac.at}

\thanks{All three authors acknowledge support by the Austrian Science Fund (FWF), Grant P29487. Nicole Vorderobermeier also acknowledges funding by the Austrian Marshall Plan Foundation. }


 
  \begin{abstract}
We prove short-time existence for the negative $L^2$-gradient flow of the $p$-elastic energy of curves via a minimising movement scheme.
In order to account for the degeneracy caused by the energy's invariance under curve reparametrisations,
we write the evolving curves as approximate normal graphs over a fixed smooth curve.
This enables us to establish short-time existence and give a lower bound on the solution's lifetime that depends only on the $W^{2,p}$-Sobolev norm of the initial data.
\end{abstract}  
  
\begin{nouppercase}
\maketitle
\end{nouppercase}
\newsectionstyle

\section{Introduction}
\noindent
For closed curves  $\gamma\colon \mathbb R / \mathbb Z \rightarrow \mathbb R^n$ in the $W^{2,p}$-Sobolev class we shall consider the energy 
 \begin{equation} \label{eqn:pEnergy}
  E (\gamma) = \frac{1}{p} \int_{\mathbb R / \mathbb Z} |\kappa|^p ds 
  + \lambda \int_{\mathbb R / \mathbb Z}   ds ,
 \end{equation}
 i.e. the sum of the $p$-elastic energy $E^{(p)}(\gamma) = \tfrac{1}{p} \int_{\mathbb R / \mathbb Z} |\kappa|^p ds$  and a positive multiple $\lambda >0$ of the length of the curve.
A family of regular curves $\gamma=\gamma(t,s)\colon  [0,T) \times \mathbb  R / \mathbb Z \rightarrow \mathbb R^n$ in the class 
$$L^\infty\big([0,T), W^{2,p}(\mathbb R / \mathbb Z, \mathbb R^n)\big) \cap W^{1,2}
\big([0,T) ,L^2 (\mathbb R / \mathbb Z , \mathbb R^n) \big)$$ 
is said to be a weak solution of the negative $L^2$-gradient flow of $E$
if one has 
 \begin{equation} \label{eq:EvolutionEquation}
  \int_{0}^T \int_{\mathbb R / \mathbb Z} \langle \partial_t \gamma ,\psi \rangle \, ds dt = \int_{0}^T \delta_{\psi_t} E (\gamma_t) \, dt
 \end{equation}
for all test functions $\psi  \in  C^\infty_c (\mathbb R / \mathbb Z \times (0,T) , \mathbb R^n)$,
i.e.~the curve $\gamma$ satisfies 
$\partial_t \gamma = - \nabla_{L^2} E (\gamma)$ weakly,
where 
 $\delta_{\psi_t} E (\gamma_t) = \frac d {d\varepsilon } \left. E (\gamma_t + 
\varepsilon \psi_t) \right|_{\varepsilon = 0} $ 
is the first variation of the functional $E$ at the curve $\gamma_t=\gamma(t, \,\cdot)$ in the direction of the test function $\psi_t=\psi(t, \,\cdot)$.

While the $L^2$-gradient flow of \eqref{eqn:pEnergy} has been extensively studied when $p=2$, both in the Euclidean (cf.~\cite{Huisken1999, Dziuk2002,DallAcqua2017, DallAcqua2019}) and manifold constrained (cf.~\cite{Huisken1999, DallAcqua2018, Mueller2019}) settings, very little is know in the degenerate $p\neq 2$ case. 
For example a second order evolution equation has been considered 
for closed curves and planar networks (cf.~\cite{Okabe2018,Novaga2020}) and the asymptotic of the flow has been studied away from degenerate point (cf.~\cite{Pozzetta2019}), 
however short-time existence for the equation \eqref{eq:EvolutionEquation} has yet to be established 
when $p\neq 2$. 
The aim of this article is to address both short and long time existence in the case $p>2$  for the geometric evolution \eqref{eq:EvolutionEquation}  with   initial data 
in the $W^{2,p}$-Sobolev class. %
Our approach is to rewrite the evolving curves as approximate normal graphs 
in order to utilise de Giorgi's method of minimising movements (cf.~\cite{DeGiorgiMM}).

It is well known that the invariance of the energy \eqref{eqn:pEnergy} under reparametrisations of the curve $\gamma$ leads to an evolution equation \eqref{eq:EvolutionEquation} that fails to be 
strongly %
parabolic (even in the $p=2$ case). 
This characteristic is in common with many other geometric evolution equations. 
For example  
the failure of the strong ellipticity of the Ricci tensor
is principally due to the second Bianchi identities.\footnote{In fact the diffeomorphism invariance of the Riemannian curvature tensor naturally yields the Bianchi identities (cf.~\cite{Kazdan1981}). Thus the strongly ellipticity failure of the Ricci tensor is due entirely to this geometric invariance.}  
For this reason, short-time existence for the Ricci flow was 
originally established in \cite{Hamilton1982a}
by appealing to the Nash-Moser implicit-function theorem 
(and the earlier exposition in \cite{hamilton1982NashMoser}). 
DeTurck \cite{DeTurck1983} subsequently showed  that the Ricci flow is equivalent to an initial value problem for a parabolic system modulo the action of the diffeomorphism group of the underlying manifold.
Thus, in a dramatic simplification that bypassed the Nash-Moser argument,  one can pass from a weakly parabolic  to a strongly parabolic system of equations
by an appropriate choice of a 1-parameter family of diffeomorphisms. 
Perelman \cite{Perelman2002} also exploited the same 
diffeomorphism invariance 
in his gradient flow formalism for the Ricci flow.
Versions of the DeTurck trick
have also been used to obtain short-time existence for the mean curvature flow 
(cf.~\cite{Huisken1984,Baker}), the Willmore flow  (cf.~\cite{Jakob2018})
and the gradient flow of the elastic energy 
in both the Euclidean and  manifold constrained cases.

In seeking to pass from the degenerate flow \eqref{eqn:pEnergy} to a strongly parabolic system,
one can consider a time dependent family of curves  
$\gamma_t = \gamma(t, \,\cdot \,)$  
that are written as a normal graph over a given fixed smooth curve 
$\tld \gamma$,  
i.e.~a family of curve of the form $\gamma_t = \tld \gamma + \phi_t$ where 
$\phi_t = \phi(t, \,\cdot\,)$ is a perturbation normal to the fixed curve $\tld \gamma$. 
In this way we obtain an evolution equation of the form
\begin{equation}\label{eq:EvolutionEquationNormal}
  \int_{0}^T \int_{\mathbb R / \mathbb Z} \langle \gamma, \partial_t^\bot  \psi\rangle\, ds dt = \int_{0}^T \delta_{\psi_t} E (\gamma_t) \, dt
\end{equation}
for all test functions $\psi  \in  C^\infty_c (  (0,T) \times \mathbb R / \mathbb Z , \mathbb R^n)$,
i.e.~the curve $\gamma$ satisfies 
$\partial_t^\bot \gamma = - \nabla_{L^2} E (\gamma)$ weakly,
where  the normal velocity
$\partial_t^\bot \gamma$ is the vector component of $\partial_t \gamma$ normal to the fixed curve $\tld \gamma$. 
Then in order to obtain a solution of  \eqref{eq:EvolutionEquation} from a solution  of \eqref{eq:EvolutionEquationNormal},
one can consider solutions $\Theta_t = \Theta(t, \,\cdot\,)$ of  the ordinary differential equation 
\begin{equation} \label{eqn:ODE} 
 \left.\begin{aligned}
       \partial_t \Theta(t,x) & = F(t,\Theta(t,x)) \\
       \Theta(0,x)  &=x ,
       \end{aligned}
 \right .
\end{equation}
where 
$F(t,y) = -\frac{\langle \partial_t \gamma (t,y), \gamma'(t,y)\rangle }{|\gamma'(t,y)|^2}$ 
and $\gamma$ is a solution of \eqref{eq:EvolutionEquationNormal}.
The existence of   ODE solutions can thus be established on a time interval $0 \leq t < \varepsilon$ for some  $\varepsilon>0$ 
independent of the initial point $x\in\R /\Z$.
Therefore if
$\Theta_t = \Theta(t, \,\cdot\,)$ is a solution of \eqref{eqn:ODE} and
$\gamma_t = \gamma(t,\,\cdot\,)$ is a solution of \eqref{eq:EvolutionEquationNormal}, the composition     
$\gamma_t \circ \Theta_t$ 
is a solution of \eqref{eq:EvolutionEquation}.
By taking this approach one can thus establish the existence of solutions for geometric flows 
with initial data in the $C^{2,\alpha}$-H\"older class 
even though  the original equations may be  ill-defined    
(see, e.g., \cite{Escher1998,Mayer2000,Simonett2001}). 
In fact a recent paper by LeCrone, Shao and Simonett \cite{LeCrone2018} showed how to reduce the regularity of the initial data to the $C^{1,\alpha}$-H\"older class. 

In order to carry out the aforementioned programme, one has to guarantee that 
a given initial curve  $\Gamma$ 
can be written as a normal graph over a fixed smooth curve $\tld \gamma$.    
Since it is not possible to write every curve $\Gamma$ in the  $W^{2,p}$-Sobolev class
as a normal graph over a smooth curve, 
we are spurred on to introduce  
the notion of a unit quasi-tangent $\tau$ 
(cf.~Definition~\ref{def:quasitangent}) 
which then defines 
an \emph{approximate tangential projection} $P_\tau^T $ and  
an \emph{approximate normal projection} $P_\tau^\bot = I  - P_\tau^T$ 
(cf.~Definition~\ref{def:phiNormalSpace}).
In which case, one can write the curve $\Gamma$ as equal to $\tld \gamma + \Phi$ up to a reparametrisation, i.e.~as an approximate normal graph over a smooth curve  
$\tld \gamma  $
with some perturbation $\Phi$ orthogonal to $\tau$ (cf.~Lemma~\ref{lem:approxNormalGraph}).
Then by applying a minimising movements scheme,
 it is  possible to establish the existence of a 
 family of curves of the form 
 $\gamma_t = \tld \gamma  + \phi_t$,
 for a suitable perturbation $\phi_t$ orthogonal to $\tau$, 
 that satisfies 
$\p_t^\bot \gamma = - \nb_{L^2} E(\gamma)$ weakly. 
Indeed, we have:

\begin{theorem}[Existence] \label{thm:weakSolution}
For any given initial curve 
$\Gamma \in W^{2,p}(\mathbb R /L \mathbb Z, \mathbb R^n)$
parametrised by arc-length
there exists
a smooth curve  $\tld \gamma   \in C^\infty(\R/L\Z,\R^n)$  parametrised by arc-length,
a quasi-tangent $\tau$ to the   curve $\tld \gamma$, %
a finite   time $T=T(p,\lambda, E(\Gamma))>0$ 
and a family of  perturbations $\phi$  
in the class 
\[
L^\infty \big([0,T),W^{2,p}(\R/L\Z,\R^n) \big) \cap 
\big(W^{1,2} \cap C^{1/ 2}\big) \big ([0,T),L^2(\R/L\Z,\R^n) \big)  
\]
which are orthogonal to $\tau$ 
such that 
the family of curves 
$$\gamma(t,s)  = \tld \gamma(s) + \phi(t,s),\quad  0\leq t<T,$$
satisfies  the initial condition 
$\gamma(0, \,\cdot\, ) = \Gamma\circ \sigma$ 
for some reparametrisation  $\sigma$ of $\R/L\Z$
and 
\begin{equation}\label{eqn:mainthm}
  \int_{0}^T \int_{\mathbb R / L\mathbb Z}
  \langle \partial_t^\bot  \gamma ,\psi\rangle\, ds dt 
  = - \int_{0}^T \delta_{\psi_t} E (\gamma_t) dt
 \end{equation} 
for all  test functions $\psi \in C^\infty_c((0,T) \times \mathbb R / L\mathbb Z   ,\R^n )$ 
 orthogonal to  $\tau$.
\end{theorem}
 
 Note that the time of existence only depends on the energy of the initial curve. So we are very close to restarting the flow and deduce long time existence. We discuss in the final section, why this is not as straightforward as it might seem.

By assuming the solution has some additional regularity, 
one can show that equation \eqref{eqn:mainthm} holds for all test functions  (i.e.~our solution solves the original weak form of the desired evolution equation).

\begin{corollary} \label{cor:weakSolution2}
If the solution $ \gamma(t, \,\cdot\, )$ of 
Theorem~\ref{thm:weakSolution} belongs to the $W^{3,p}$-Sobolev class for almost all $0\leq t<T$, then
$$
 \int_0^T \int_{\mathbb R / L\mathbb Z}  
 \langle \partial^\bot_t  \gamma ,\psi\rangle ds dt 
 = - \int_0^T \delta_{\psi_t} E(\gamma) dt
$$
for all test functions $\psi \in C^\infty_c ((0,T) \times \R /L \Z , \R^n)$.
\end{corollary}

\section{Minimising movements scheme}
\noindent
It is remarked by De Giorgi \cite{DeGiorgiMM}
that a generalised minimising movements scheme could provide a formalism for
the existence of steepest descent curves of a functional in a metric space. 
In order to establish the existence of weak solutions for \eqref{eq:EvolutionEquation},
we need to 
take care of the
twofold degeneracies arising from the 
invariance of \eqref{eqn:pEnergy} under curve reparametrisation and the fact that $p>2$.
We tackle this issue 
by writing the evolving curve as an approximate normal graph over a fixed smooth curve
so that we can work with 
the normal velocity (rather than the time derivative) of 
the evolving curve.

\subsection{Tubular neighbourhoods} \label{sec:Tubenbd}
For an embedded $C^k$-submanifold  $\mathcal{M}$ of $\R^n$  without boundary,
the normal bundle $(T\mathcal{M})^\perp \to \mathcal{M}$ is only of the class $C^{k-1}$. 
If we define the `endpoint' map $E\colon (T\mathcal{M})^\perp \to \R^n$   by sending $$(x,v) \mapsto x+v$$ and assume  $k\geq 2$,  
one can use the inverse function theorem to show that there exists 
a tubular neighbourhood $U$ of $\mathcal{M}$ in $\R^n$
that is the diffeomorphic image under the $C^{k-1}$-map $E$ of an open neighbourhood of the zero section of  $(T\mathcal{M})^\perp$.
Moreover, 
the squared distance function $\zeta (x) = \frac 12 \text{dist}(x,\mathcal{M})^2$
is a function in  $C^{k}(U)$ (cf.~\cite{Foote}) 
and the Hessian matrix  $\nb^2 \zeta(x)$ represents the orthogonal projection on the normal space to $\mathcal{M}$ at  a point $x$ (cf.~\cite[p.~704]{ambrosio1996}). 
 Of course these results
 no longer hold in the case $k=1$, i.e.~when the inverse function theorem is not applicable.

\subsection{Approximate normal graphs} \label{sec:NormalGraphs}

As the normal bundle of an embedded $W^{2,p}$-curve  in $\R^n$ is only of the class $W^{1,p}$, one cannot  directly apply
the standard methods
of \S\ref{sec:Tubenbd}. 
In particular, we need to overcome the loss of regularity on the level of the tangent  space
in order to write the solution of our equation locally as a graph over a fixed smooth curve.
This problem  can be resolved by  
regularising the tangent
using Friedrichs mollifiers
(whilst taking into consideration the size of the constructed tubular neighbourhood). We will call this smoothened tangent \emph{quasi-tangent}.

\begin{definition}
A function $\eta \in C^\infty(\R)$ is called a  \emph{mollifier} if it satisfies the conditions:
(i) $\eta \geq 0$ on $\R$, 
(ii) $\eta (x) = 0 $ for all   $|x| \geq 1$,
and (iii) $\int_\R \eta(x) dx  = 1$. 
The associated \emph{rescaled mollifier} is the function  
$\eta_\varepsilon (x) = \frac 1 \varepsilon \eta(\frac x \varepsilon)$ for any $\ep>0$.
\end{definition}

Now consider a curve $\gamma \in W^{2,p}(\mathbb R / \mathbb Z, \mathbb R^n)$ parametrised by arc-length. 
The \emph{mollification} of $\gamma$ is defined to be the function 
$$\gamma_\ep(x) = (\gamma \ast \eta_\varepsilon )(x) = \int_\R \gamma(x-y)   \eta_\varepsilon (y) dy ,$$
i.e.~the convolution of the given curve $\gamma$ and the rescaled mollifier $\eta_\varepsilon$. 

For the mollified curve $\gamma_\ep$  we  derive the following  well-known estimates. 
Firstly, from the mean value theorem and the Sobolev embeddings, we find that 
\begin{align}
| \gamma_\varepsilon(x) - \gamma(x) | 
& = \left| \int_{\mathbb R} (\gamma(x-y) - \gamma(x) )  \eta_\varepsilon (y)  dy\right| \nonumber \\
& \leq \varepsilon \|\gamma'\|_{L^\infty} \nonumber \\
& \leq C  \varepsilon \|\gamma'\|_{W^{1,p}}  \label{eq:SmoothendCurveUniformEstimate} 
\end{align}
Likewise, we   find that 
\begin{align} 
| \gamma'_\varepsilon(x) - \gamma'(x) | 
& = \left| \int_{\mathbb R} \eta_\varepsilon (y) (\gamma'(x-y) - \gamma'(x) )  dy\right| \nonumber \\ 
& \leq \sqrt{\ep}  \|\gamma'\|_{C^{1/2}} \nonumber \\ 
& \leq C  \sqrt{\ep} \|\gamma'\|_{W^{1,p}}. \label{eq:SmoothendCurveUniformEstimateDerivative}
\end{align}
For higher derivatives we can use the  Sobolev embeddings, H\"older's inequality and integration by parts 
to obtain the $L^\infty$-bound
\begin{align} 
| \gamma^{(k+2)}_\varepsilon(x) | 
& =\left| \int_{\mathbb R}  \eta_{\varepsilon}^{(k)}(y) \gamma''(x-y) dy\right|  \nonumber \\ 
& \leq  \|\eta^{(k)}_{\varepsilon} \|_{L^q} \|\gamma''\|_{L^p}  \nonumber \\ 
&\leq C  \varepsilon^{-k-1+\frac 1 q} \|\gamma''\|_{L^p}   \nonumber \\ 
&= C  \varepsilon^{-k - \frac 1 p} \|\gamma''\|_{L^p} \label{eq:SmoothendCurveCurvatureEstimateHigherDerivaitves}
\end{align}
for integers $k\geq 0$ with $\frac 1p + \frac 1q = 1$.

We will use the next lemma to fix the smoothing parameter $\varepsilon.$

\begin{lemma} \label{lem:ApproximateTangent}
If for an $M> 0$ we have a curve $\gamma \in W^{2,p}(\mathbb R / \mathbb Z, \R^n)$ parametrised by arc-length which satisfies
$
\|\gamma'\|_{W^{1,p}} \leq M  
$,
then there exists an $\ep =\ep(p,M) >0$ 
such that the  unit tangent
$
\tau = \frac { \gamma'_{\varepsilon}} {|\gamma'_{\varepsilon}|}
$
satisfies   
\begin{equation} \label{eq:ApproximateTangent}
 \|\tau -  \gamma' \|_{L^\infty} \leq \frac 1 4
.\end{equation}
\end{lemma}

\begin{proof}
Using \eqref{eq:SmoothendCurveUniformEstimateDerivative} 
we get 
$
 \|\gamma_{\varepsilon}' - \gamma' \|_{L^\infty} \leq C \sqrt \ep \|\gamma'\|_{W^{1,p}}.
$								
As the retraction map $\Pi\colon x \mapsto \frac x {|x|}$ is locally Lipschitz on $\R^n\backslash \{0\}$, 
the    tangent $\tau = \Pi (\gamma'_{\varepsilon}) =  \frac { \gamma'_{\varepsilon}}  {|\gamma'_{\varepsilon}|}$  to the mollified curve $\gamma_{\ep}$ 
  satisfies 
\eqref{eq:ApproximateTangent} for some     $\ep>0$ 
sufficiently small. 		
\end{proof}

\begin{corollary} \label{cor:DerivativeAprTangent}
If for an $M> 0$ we have a curve $\gamma \in W^{2,p}(\mathbb R / \mathbb Z, \R^n)$ parametrised by arc-length which satisfies
$
\|\gamma'\|_{W^{1,p}} \leq M  
$,
then for an $\ep=\ep(p,M)>0$ as in Lemma \ref{lem:ApproximateTangent} 
the mollified curve $\gamma_\ep$
 has a
  unit tangent map
$\tau    \colon \mathbb R / \mathbb Z \to S^{n-1}  $  
that 
is smooth 
and    satisfies
\begin{equation} \label{eq:EstApproximateTangent}
 \|\tau'\|_{L^\infty} , \|\tau''\|_{L^\infty} \leq C 
\end{equation}
for a constant   $C=C(p,M)>0$. 
\end{corollary} 

\begin{definition} \label{def:quasitangent}
We say $\tau$ is
\emph{unit quasi-tangent} to the $W^{2,p}$-curve $\gamma$
if it is the unit tangent to the mollified curve $\gamma_\ep$ for some $\ep=\ep(p,M)>0$ as in Lemma \ref{lem:ApproximateTangent}. 
\end{definition}

\begin{definition}
We denote by 
$P^T _v w = \langle w, \frac v {|v|}  \rangle \frac v {|v|}$
the orthogonal projection of $w$ onto the line $\mathbb R v$
for any  vectors $v,w \in \mathbb R^n$.  
Likewise,  we denote by
$P^\bot _v w = w - P^T _v w$
the orthogonal projection of $w$ onto the orthogonal complement $(\mathbb R v)^\bot$  of the line $\mathbb R v$.
\end{definition}

\begin{definition} \label{def:phiNormalSpace}
If $\tau$ is unit quasi-tangent to a $W^{2,p}$-curve $\gamma$,
we denote by 
$(W^{2,p})^T_{\tau}$  (resp.~$(W^{2,p})^\bot_{\tau}$)
the set of all $w \in W^{2,p}(\mathbb R / \mathbb Z, \R^n)$ such that 
$P_\tau^\bot w =0$ a.e. (resp.~$P_\tau^T w= 0$ a.e.). 
\end{definition}

We will now prove  the following statement that gives a lower bound on the thickness of the set of regular curves around $\gamma$.
 
 \begin{lemma} \label{lem:RegularCurve}
If for an $M> 0$ we have a curve $\gamma \in W^{2,p}(\mathbb R / \mathbb Z, \R^n)$ parametrised by arc-length which satisfies
$
\|\gamma'\|_{W^{1,p}} \leq M  
$,
then there exists  a constant $K=K(p,M)>0$  %
and a unit quasi-tangent $\tau$ to the curve  $\gamma$ 
 such that 
 	the curve $\gamma + \phi$ satisfies 
\[
 \inf_{x\in \mathbb R / \mathbb Z}\langle \gamma' + \phi', \tau \rangle \geq  \frac 1 2
\]
and hence
\[
\inf_{x\in \mathbb R / \mathbb Z} |\gamma'(x) + \phi'(x)| \geq \frac 1 2
\]
for each $\phi \in (W^{2,p})^\bot_{\tau}$ with 
 $
 	 \|\phi\|_{L^\infty} \leq K 
 $.
In particular, $\gamma+\phi$ is a regular curve.
 \end{lemma}

 \begin{proof}
We first note that
$
\langle \gamma', \tau \rangle = |\gamma'|^2 + \langle \gamma', \tau - \gamma'\rangle
\geq 1 - |\tau - \gamma'| \geq \frac 3 4
$
by   
Lemma~\ref{lem:ApproximateTangent}
and the fact that  $|\gamma'|=1$. 
Upon differentiating the orthogonality condition $\langle \phi , \tau \rangle = 0 $, we get $\langle \phi', \tau \rangle = - \langle \phi, \tau' \rangle$.
In which case the estimate \eqref{eq:EstApproximateTangent} implies that 
 $
  |\langle \phi', \tau\rangle | = |\langle \phi, \tau'\rangle| 
  \leq C %
  \| \phi \|_{L^\infty } \leq \frac 1 4
 $
 whenever 
 $\|\phi\|_{L^\infty} \leq  
 \frac {1}{4 C } %
 = K$. 
 Thus %
 $$
 \langle \gamma' + \phi', \tau \rangle \geq \frac 3 4 - \frac 1 4 = \frac 1 2
 $$ whenever $\|\phi\|_{L^\infty} \leq K $, i.e.~$\gamma + \phi$ is a regular curve.
 As $\tau$ is of unit length, we also have
 $
 |\gamma' + \phi'| \geq \frac 1 2
 $ on $ \mathbb R / \mathbb Z$.
 \end{proof}

We will now deduce the following lower bound for the $L^p$-norm of the curvature of a curve  $\tld \gamma+\phi$ 
in terms of the
$L^p$-norm of the second derivative of $\phi$. 
This bound extends to our situation 
the well know analogous result for the case of a real normal graph over a smooth curve.

\begin{lemma} \label{lem:EstSecDerivative}
If for an $M> 0$ we have a curve $\gamma \in W^{2,p}(\mathbb R / \mathbb Z, \R^n)$ parametrised by arc-length which satisfies
$
\|\gamma'\|_{W^{1,p}} \leq M  
$,
then for the constant $K=K(p,M)>0$ from Lemma \ref{lem:RegularCurve}  %
and a unit quasi-tangent $\tau$ to the curve  $\gamma$ 
 such that for each $\phi \in (W^{2,p})^\bot_{\tau}$ with 
 $
 	 \|\phi\|_{L^\infty} \leq K 
 $
$$
| v| \leq C |P_{ \gamma'+ \phi'}^{\bot} v|
$$
for all $v\in \R^n$  pointing in an approximate normal direction and 
 $$
  \int_{\R/\Z} | \phi''|^p ds  \leq C 
 \bigg( 1+  \int_{\R/\Z} |\kappa_{ \gamma + \phi}|^p ds   \bigg)
 $$
for some $C=C(M,p)$. 
\end{lemma}

\begin{proof}
Since
$
\langle  \gamma' + \phi', \tau \rangle \geq \frac 12
$
from Lemma \ref{lem:RegularCurve}
and 
$| \gamma ' + \phi'| \leq | \gamma'|+|\phi'| \leq 1+\Lambda$,
we see that 
$
 \langle \frac { \gamma' +  \phi'}{| \gamma' + \phi'|} , \tau \rangle
 \geq \frac 12 \frac 1 {1+\Lambda} .
$
Hence the angle between $ \gamma' + \phi'$ and $\tau$ is bounded strictly away from $\frac \pi 2$.
In which case we have
$$
|v| \leq C |P_{ \gamma'+ \phi'}^{\bot} v|
$$
for all $v \in \mathbb R^n$ pointing in an approximate normal direction.

For the second estimate, we recall the curvature formula given by 
\[
\kappa_{ \gamma+ \phi} = \frac {P^\bot_{ \gamma' + \phi'} ( \gamma'' + \phi '')}{| \gamma' + \phi'|^2}
.\]
Now by the triangle inequality we see that 
\begin{align*}
\left| P^\bot_{ \gamma' + \phi'} ( \phi '' ) \right|
\leq \left| P^\bot_{ \gamma' + \phi'} ( \gamma''+ \phi '' ) \right|
+ \left| P^\bot_{ \gamma' + \phi'} ( \gamma''  ) \right| 
\leq C (|\kappa_{ \gamma+ \phi}| + |\gamma ''|) ,
\end{align*}
since $| \gamma' + \phi'|\leq | \gamma'| + |\phi'|\leq 1 +\Lambda$.
To control the tangential part  $P^T_\tau \phi''$, we  differentiate  the equation 
$\langle \phi, \tau \rangle=0$ twice to get
$
 \langle \phi '' , \tau \rangle = - 2 \langle \phi', \tau ' \rangle - \langle \phi, \tau '' \rangle.
$
It then follows that  
\begin{align*}
|P^T_{\tau}  \phi ''| = | \langle \phi '' , \tau \rangle|
\leq  | \langle \phi, \tau '' \rangle| +  2 |\langle \phi', \tau ' \rangle|
\leq C (K+ \Lambda) ,
\end{align*}
since both $\tau'$ and $\tau''$ are bounded by Corollary \ref{cor:DerivativeAprTangent}.
In combining both the tangential and normal parts of $\phi''$ and using the fact that the angle between $ \gamma' + \phi'$ and $\tau$ is bounded strictly away from $\tfrac \pi 2$, we find  that 
\[
|\phi''| \leq C( |P^\bot _{\gamma' + \phi'}  \phi ''| + |P^T_{\tau}  \phi ''| )
\leq C (1+  | \gamma'' | + |\kappa_{ \gamma+ \phi}|)
\]
from which the desired integral estimate follows (since
$\|  \gamma''\|_{L^p} \leq M$ by Lemma~\ref{lem:approxNormalGraph}.
\end{proof}

 Next we show that there exists a good substitute for the nearest neighbourhood projection which yields a local tubular neighbourhood. %
 We also obtain a lower bound on  thickness of the tubular neighbourhood  that   only depends   on the $W^{2,p}$-norm of the curve. 
 
 \begin{definition} 
If $\tau$ is a unit quasi-tangent to a $W^{2,p}$-curve $\gamma$,
the $(n-1)$-dimensional subspace 
$$\mathcal{N}_{x_0}=\{v \in \mathbb R^n: P^T_{\tau(x_0)} v =0\}$$
 is called an  \emph{approximate normal space} 
to $\gamma$  %
at a given fixed  point  $x_0 \in \mathbb R / \mathbb Z$. 
 \end{definition}

By considering the  map
$  H_{x_0} \colon  B_{\delta}(x_0) \times \mathcal{N}_{x_0} \rightarrow \mathbb R^n$
 given by 
\begin{equation}\label{eqn:Hmap}
(x,v) \mapsto \gamma(x) + P^\bot_{\tau(x)} v
\end{equation} 
for some $0<\delta<1$, we obtain the following:
 
 \begin{lemma} \label{lem:Tubularneighbourhood}
If for an $M> 0$ we have a curve $\gamma \in W^{2,p}(\mathbb R / \mathbb Z, \R^n)$ parametrised by arc-length which satisfies
$
\|\gamma'\|_{W^{1,p}} \leq M  
$,
  then there exists a  sufficiently small constant $\delta = \delta(p,M)   >0$ %
and a unit quasi-tangent $\tau$ to the curve $\gamma$  
  such that     \eqref{eqn:Hmap}
    maps $B_\delta (x_0) \times B_\delta (0)$ diffeomorphically onto its image and 
  \begin{equation}\label{eqn:Hballs}
    B_{  \delta/ 4} \big(\gamma(B_{  \delta / 4}(x_0))\big) 
    \subset H_{x_0} \big(B_\delta (x_0) \times B_\delta (0)\big)
   .\end{equation} 
 \end{lemma}
 
 \begin{proof}
 We first show that $H_{x_0}$ is a local diffeomorphism by way of  the inverse function theorem. To do so we calculate the partial derivatives %
\begin{align*}
\frac {\p H_{x_0}} {\p x}   & = \gamma'(x) - \langle v, \tau'(x) \rangle \tau(x) - \langle v, \tau(x) \rangle \tau'(x) \\
\frac {\p H_{x_0}} {\p v}      & = v   + P^\bot_{\tau (x)} v - P^\bot_{\tau (x_0)} v
.\end{align*}
Then from the estimates \eqref{eq:ApproximateTangent} and \eqref{eq:EstApproximateTangent} together with the Sobolev embedding  $W^{2,p}(B_\delta(x_0),\R^n) \hookrightarrow C^{1,1-\frac 1p} (B_\delta(x_0),\R^n)$
we find that 
 \begin{align*}
 \Big|  \frac{\p H_{x_0}} {\p x}  - \tau(x_0) \Big| 
  & \leq  | \gamma'(x_0) - \tau(x_0)| +  C|v| +|\gamma'(x) - \gamma'(x_0)|   
 \\ & \leq  \frac 1 4 + C |v|  + C  \delta ^{1-\frac 1 p}
\end{align*}
for some constant $C=C(p,M)>0$. %
By taking some $\delta  > 0$ sufficiently small (depending only on $p$ and $M$), we have 
$$
\Big |  \frac{\p H_{x_0}}{\p x}  - \tau(x_0) \Big| \leq \frac 12 
$$
for all $x\in B_\delta(x_0) \subset \R/\Z$ and   $v \in B_\delta(0)\subset \mathcal{N}_{x_0}$.
Likewise, whenever  $\delta  > 0$ is sufficiently small, 
we also have 
$$\Big | \frac{\p H_{x_0}}{\p v}   - v \Big| \leq \frac 1 2$$
for all $(x,v) \in B_\delta(x_0) \times  B_\delta(0)$.

 Let us now assume that $\tau(x_0) = e_1$ without loss of generality. 
From the above estimates 
we see that the Jacobi matrix $ D H_{x_0}$  satisfies
 \begin{equation} \label{eq:LocalDiffeo}
  \|D H_{x_0} - I\| \leq \frac 12 ,
 \end{equation}
 where $\|\cdot\|$ denotes the operator norm. Therefore $DH_{x_0}$ is invertible and 
 so $H_{x_0}$ maps $B_{\delta}(x_0) \times B_{\delta}(0)$ diffeomorphically onto its
image by the inverse function theorem.
 Moreover,    \eqref{eq:LocalDiffeo}  implies that 
 \begin{align*}
  |H_{x_0}(z_1)- H_{x_0}(z_2)| & = \left| \int_{0}^1  D H_{x_0} (z_2 + \theta (z_1 - z_2)) (z_1-z_2) d\theta \right| \\
  &\geq |z_1 - z_2 | - \tfrac 1 2 |z_1 - z_2|  \\
&  = \tfrac 1 2 |z_1 - z_2| .
 \end{align*}
In which case the map $H_{x_0}$ is    bi-Lipschitz and hence injective on $B_{\delta}(x_0) \times B_{\delta}(0)$. 
From the fact that 
  $
  \dist  ( \partial (B_{\delta} (x_0) \times B_{\delta} (0)),  B_{ \delta/ 2} (x_0) \times \{0\} ) \geq  \frac \delta  2 
  $
and the latter bi-Lipschitz estimate   
   we have 
   $$
  \dist \Big(
H_{x_0}\big(\partial (B_{\delta} (x_0) \times B_{\delta} (0))\big), 
 \gamma(B_{  \delta /2} (x_0))\Big)   
  \geq  \frac  \delta  4
 $$
which then gives  %
\eqref{eqn:Hballs}.
 \end{proof}

We can now use  Lemma \ref{lem:Tubularneighbourhood} to show that 
any $W^{2,p}$-curve $ \gamma$ can be written as an approximate normal graph over a given $W^{2,p}$-curve $\tld \gamma$
whenever the curves are     $C^1$-close to each other.

\begin{lemma} \label{lem:ApproxNormalGraphs}
If for an $M> 0$ we have a curve $\gamma \in W^{2,p}(\mathbb R / \mathbb Z, \R^n)$ parametrised by arc-length that satisfies
$
\|\gamma'\|_{W^{1,p}} \leq M  
$,
then there exists a sufficiently small constant $\rho=\rho(p,M) >0$ %
and a unit quasi-tangent $\tau $   to the curve $\gamma$ 
such that for each curve   $\tld \gamma \in W^{2,p}(\mathbb R / \mathbb Z, \R^n)$ satisfying $\| \gamma - \tld \gamma\|_{C^1} \leq \rho$
we have  
some $\phi \in (W^{2,p})^\bot_\tau$
and a reparametrisation $\sigma$ of $\R/ \mathbb Z$
for which 
$
\tld \gamma \circ \sigma = \gamma + \phi
.$
\end{lemma}

\begin{proof}
Firstly, choose $\delta >0$  as in Lemma \ref{lem:Tubularneighbourhood}. Since $\mathbb R / \mathbb Z$ is compact, there exists points  $x_1$, \ldots, $x_\ell$  in  $\mathbb R / \mathbb Z$ such that the balls $B_{  \delta/ 4 } (x_1)$, \ldots, $B_{  \delta/ 4 } (x_\ell)$  cover $\mathbb R / \mathbb Z$. 
 Let the mappings  $H_{x_j}$  for $j=1,\ldots,\ell$  be  defined by \eqref{eqn:Hmap} and let
$
 \Pi_{x_j}\colon  
 B_{  \delta /  4} \big( \gamma(B_{  \delta  / 4} (x_j ) ) \big) \rightarrow \mathbb R  /\mathbb Z 
$
be  the corresponding retraction maps given by 
\[
\Pi_{x_j} =  \pi \circ  H_{x_j}^{-1}
\]
where $\pi\colon  B_{\delta} (x_j) \times \mathcal{N}_{x_j} 
\rightarrow \mathbb R / \mathbb Z$ 
sends $(x,v) \mapsto x$, i.e.~the projection onto the first coordinate.
We can then set 
\begin{equation}\label{eqn:sigmadef}
\sigma(x) = \Pi_{x_j} (  \tld \gamma(x))
\end{equation} 
for any  $x \in B_{\delta} (x_j)$  
 in order to get a well-defined $C^1$-mapping. 
Furthermore, 
 from 
 the inverse function theorem applied to $\Pi_{x_j}$
 and the estimate \eqref{eq:LocalDiffeo}
 we see that 
  $
  \sigma'(x) >0
 $ whenever $\rho>0$ is sufficiently small (i.e.~$\sigma$ is bi-Lipschitz).
In addition, 
by setting  $\tld \phi = \tld \gamma -  \gamma \circ \sigma$
we see from \eqref{eqn:Hmap}  that $\tld \phi$ belongs to $(W^{2,p})^\bot_\tau$.  
Therefore    $  \gamma \circ \sigma$ is a regular curve equal to $\tld \gamma - \tld\phi$. 
In order to change the roles of $\gamma$ and $\tld \gamma$, we apply the inverse function theorem to $\sigma$ to justify the reparametrisation $\tld \gamma \circ \sigma^{-1} = \gamma \circ \sigma \circ \sigma^{-1} + \tld \phi \circ \sigma^{-1} = \gamma + \phi$, where we set $\phi = \tld\phi \circ \sigma^{-1} \in  (W^{2,p})^\bot_\tau$. 

\end{proof}

Using the above lemmata we can write every $W^{2,p}$-curve $\gamma$ as an approximate normal graph over a smooth curve $\tld \gamma$. Be aware that from now on till the end of this article we consider normal graphs over the curve $\tilde \gamma$ instead of $\gamma.$

\begin{lemma} \label{lem:approxNormalGraph}
Let $\gamma \in W^{2,p}(\mathbb R / \mathbb Z, \R^n)$ be a curve parametrised by arc-length.
For every 
$ \ep_0>0 $  there exists 
a smooth curve $\tld \gamma \in C^\infty(\mathbb R / \mathbb Z, \R^n)$ parametrised by arc-length 
with $\|\gamma - \tld \gamma\|_{W^{2,p}} \leq \ep_0$,
a unit quasi-tangent $\tau $ to the curve $\tilde\gamma$
and  
some  $\phi \in (W^{2,p})^\bot_\tau$
such that 
\begin{equation}\label{eqn:decompositionANG}
 \gamma \circ \sigma = \tld \gamma + \phi 
\end{equation} 
for  
a reparametrisation $\sigma$ of $\R/ \mathbb Z$.
\end{lemma}

\begin{proof}
Firstly, there exists a smooth curve $\tld \gamma \in C^\infty(\R/\Z,\R^n)$ parametrised by arc-length such that 
$$\| \tilde \gamma - \gamma\|_{W^{2,p}} \leq \ep_0$$
by the density of $C^\infty(\R/\Z,\R^n)$ in $W^{2,p}(\R/\Z,\R^n)$. 
Moreover, we have 
$$ 
\| \gamma -\tld \gamma\|_{C^1} \leq C \ep_0 = \rho
$$  by the Sobolev embeddings.
Thus by taking some $\ep_0>0$ sufficiently small, 
  Lemma~\ref{lem:ApproxNormalGraphs} implies that 
there exists some 
$\phi \in (W^{2,p})^\bot_\tau$
and a reparametrisation $\sigma$ of $\R/ \mathbb Z$
such that 
$
 \gamma \circ \sigma = \tld \gamma + \phi 
$. 
\end{proof} 

The representation of $\gamma$ by a normal graph $\phi$ over $\tld \gamma$ we obtain from Lemma \ref{lem:approxNormalGraph} satisfies the following $C^1$ estimates. These enable us to control the second derivative of $\phi$ by the curvature of $\gamma$ using Lemma \ref{lem:EstSecDerivative}. 

\begin{corollary} \label{cor:smallnesscondition}
For the decomposition \eqref{eqn:decompositionANG}
there exists a constant $C>0$ depending on an upper bound $M$ on $\|\gamma'\|_{W^{1,p}}$ and $p$ 
such that 
\[
\|\phi\|_{L^\infty} \leq C \|  \gamma - \tld \gamma \|_{L^\infty}
\quad \text{and} \quad 
\|\phi'\|_{L^\infty} 
\leq C (1+ \|   \gamma' - \tld \gamma' \|_{L^\infty} )
.\]
\end{corollary}

\begin{proof}
From the construction of $\sigma$ given by \eqref{eqn:sigmadef} we see that
\[
	|\sigma(x) - x| = |\sigma(x) - \sigma \circ \sigma^{-1} (x)| \leq \|\sigma'\|_{L^\infty} |x-\sigma^{-1}(x)|
\] 
and 
\[
|x-\sigma^{-1}(x)|
= |\Pi_{x_j}(\tld \gamma(x)) - \Pi_{x_j}(  \gamma(x)) |
\leq \big( \max_j \|D \Pi_{x_j}\|_{L^\infty} \big) |\tld\gamma(x) - \gamma(x) |
,\]
since there exists some ball $B_{\delta}(x_j)$ such that $x = \Pi_{x_j}(  \tld\gamma(x))$. 
As  we have 
\[
|\phi(x)| = | \gamma (\sigma(x)) - \tld \gamma(x) | 
 \leq | \gamma (\sigma(x)) - \gamma (x) |  + | \gamma (x) - \tld \gamma(x) | 
\]
and 
$| \gamma (\sigma(x)) - \gamma (x) |   \leq \|\gamma'\|_{L^\infty} |\sigma(x) - x|$,
it follows  that 
\[
\|\phi\|_{L^\infty} 
\leq (1+ C \| \gamma'\|_{W^{1,p}} )\|  \gamma - \tld\gamma \|_{L^\infty} 
\]
by the Sobolev embeddings. %
In addition, we  have  
\begin{align*}
 \|\phi'\|_{L^\infty}  = \| (\gamma \circ \sigma) ' - \tld \gamma' \|_{L^\infty} 
& \leq 
\| \gamma' - \tld \gamma' \|_{L^\infty} 
+ \|\gamma' \|_{L^\infty}
+ \|(\gamma \circ \sigma)'\|_{L^\infty}  
 \\ & \leq   \|  \gamma' -\tld \gamma'\|_{L^\infty} + C \| \gamma'\|_{W^{1,p}}  
\end{align*}
from the uniform bi-Lipschitz property of $\sigma$ 
and the Sobolev embeddings.
\end{proof}

 \subsection{Existence of discrete-time approximations} \label{sec:Existsdiscretetime}
After breaking 
the reparametrisation invariance of \eqref{eqn:pEnergy}
by way of the approximate normal graphs,
it is now a straight forward matter to prove the short-time existence of solutions for the minimising movement scheme.

Let us first consider 
an initial curve $\Gamma \in W^{2,p}(\R/L\Z,\R^n)$ of length $L$   parametrised by  arc-length. In the following it will be essential that all estimates only depend on an upper bound on the energy of this curve. 

We first note that an upper bound on the energy also implies a lower bound on the length,
since by Fenchel's theorem together with H\"older's inequality we have 
$$
 2 \pi \leq \int_{\R / L \Z} |\kappa| ds \leq  L ^{1-\frac 1p } \bigg(\int_{\R / L \Z} |\kappa|^p ds \bigg )^{\frac 1p}
$$
so that 
$$
 L^ {p-1} \geq \frac {(2 \pi)^p} {p \,  E^{(p)}(\Gamma)}.
$$
By scaling the results of Section~\ref{sec:NormalGraphs},
we can drop the assumption that the curve is of unit length
and recover all the previous estimates concerning approximate normal graphs (with proviso that the relevant constants now depend on $\lambda$ and the energy bound). 
In particular, 
we say that the unit vector field $\tau$ is quasi-tangent to a $W^{2,p}$-curve $\gamma$ of length $L$ 
whenever  $\tau(\frac \cdot L)$ is quasi-tangent to the curve $\gamma (\frac \cdot L)$.

Now   for the initial curve, the result of 
Lemma~\ref{lem:approxNormalGraph}
implies that 
there exists 
a smooth curve $\tld\gamma$ parametrised by arc-length,
a unit quasi-tangent $\tau$ to the curve $ \tld \gamma$  
and a perturbation $\Phi \in (W^{2,p})_\tau^\perp$
such that 
$\Gamma \circ \sigma = \tld \gamma + \Phi$.
Moreover, by combining the norm bounds of Lemma~\ref{lem:approxNormalGraph} with Corollary~\ref{cor:smallnesscondition} and the Sobolev embeddings, we see that  
$$
\quad \|\Phi \|_{L^\infty} \leq \mu 
\quad \text{and}\quad 
\|\Phi'\|_{L^\infty} \leq W
$$ 
for some sufficiently small constant $\mu=\mu(p,\lambda,E(\Gamma))>0$
and some constant $W = W (p,\lambda,E(\Gamma))>2$.

For a series of discrete time steps $0=t_0< t_1 < t_2 < \cdots$
we seek to define the curves 
\begin{equation}\label{eqn:timeStep}
\gamma_{t_j} = \tld \gamma + \phi_{t_j}
\end{equation}
with the initial case $\gamma_{t_0} = \tld \gamma + \Phi$.
The time differences 
$t_{j+1} - t_j = h$  are set to be equal to a fixed parameter $h>0$ (that we shall ultimately send to zero). 
We want to recursively define  $\phi_{t_{j+1}}$ for the next time step as the minimiser 
$$ %
 \phi_{t_{j+1}} = \underset{\phi \,\in \, \mathscr{V}}{\argmin} \bigg \{ E(\tld \gamma+\phi) + \frac 1 {2h} \int_{\mathbb R / L \mathbb Z} 
 | P^{\bot}_{\gamma_{t_j}'}( \tld \gamma + \phi- \gamma_{t_j})  |^2 |\gamma_{t_j}'|dx
 \bigg\} ,
$$%
where the class of admissible perturbations is given by 
 $$
\mathscr{V} = \mathscr{V}(\mu,W) =  \{\phi \in (W^{2,p})_\tau^\bot : \|\phi \|_{L^\infty} < 3 \mu, \|\phi '\|_{L ^\infty} < 3 W\}.
 $$
The following lemma states that these discrete-time solutions can be constructed for at least a short time.  
\begin{lemma}
There exists
a finite   time $T>0$ depending only on $p$, $\lambda$ and $E(\Gamma )$ 
such that the  solutions 
$
\gamma_{t_j} = \tld \gamma + \phi_{t_j}
$
exist for a series of discrete times 
$0 = t_0 < t_1< t_2< \cdots <t_N<T $.
\end{lemma}

\begin{proof}
We seek to establish the existence of 
the perturbations 
$\phi_{t_{j+1}}$ 
that are minimisers of the functionals 
\[
\mathcal{F}_j (\phi) = 
E(\tld \gamma+\phi) + \frac 1 {2h} \int_{\mathbb R / L \mathbb Z} 
 | P^{\bot}_{\gamma_{t_j}'}( \tld \gamma + \phi- \gamma_{t_j})  |^2 |\gamma_{t_j}'|dx
\]
over the admissible class $\mathscr{V}$.
To do so we proceed by an induction argument with an initial base case 
$\phi_{t_0} = \Phi$ given by the decomposition of the initial curve $\Gamma$. 
Indeed, let us assume there exist minimisers $\phi_{t_{i+1}}$
of $\mathcal{F}_{i}$ over the class $\mathscr{V}$
for $i=0,1,\ldots,j-1$.

Now as $\mathcal{F}_i (\phi_{t_{i+1}}) \leq \mathcal{F}_i (\phi_{t_{i}})$ 
for $i=0,1,\ldots,j-1$ (i.e.~$\phi_{t_i}$ is  a competitor), we note that
\[
 E(\gamma_{t_j}) \leq E(\gamma_{t_0}) = E (\Gamma)
\]
and
\[
\frac 1 {2h} \int_{\mathbb R / L \mathbb Z} 
 | P^{\bot}_{\gamma_{t_i}'}( \gamma_{t_{i+1}} - \gamma_{t_i})  |^2 |\gamma_{t_i}'|dx
 \leq E(\gamma_{t_i}) - E(\gamma_{t_{i+1}}).
\]
In which case Lemma~\ref{lem:EstSecDerivative} implies that 
the $L^p$-norm of $\gamma''_{t_j}$ is uniformly bounded by a constant which depends only on $p$ and $E(\Gamma)$. In addition, we   have  
\[
\frac 1 {  h} \int_{\mathbb R /L \mathbb Z} 
 |   \gamma_{t_{i+1}} - \gamma_{t_i}  |^2  dx
 \leq C \big(  E(\gamma_{t_i}) - E(\gamma_{t_{i+1}}) \big)  .
\]
Then by summing up the latter inequalities, we get the a priori estimate 
\begin{equation} \label{eq:APriori_I}
\sum_{i=0}^{j-1} \frac 1 { h} \int_{\mathbb R / L \mathbb Z} 
 |   \gamma_{t_{i+1}} - \gamma_{t_i}  |^2  dx
 \leq  C \big( E(\gamma_{t_0}) - E(\gamma_{t_{j}})   \big)
.\end{equation}
We also recall from H\"older's inequality that
\begin{equation} \label{eq:APriori_II}
\begin{aligned}
\| \gamma_{t_0} - \gamma_{t_{j}} \|_{L^2} 
& \leq \sum_{i=0}^{j-1}
 \frac{\| \gamma_{t_{i+1}} - \gamma_{t_i} \|_{L^2} }{\sqrt h}  \sqrt{h}\\
& \leq 
\left ( \sum_{i=0}^{j-1} \frac 1 { h} \int_{\mathbb R / L \mathbb Z} 
 |   \gamma_{t_{i+1}} - \gamma_{t_i}  |^2  dx\right )^{\frac 12} 
 \left ( \sum_{i=0}^{j-1} h \right )^{\frac 12 } 
 \\
 &\leq C \sqrt{E(\gamma_{t_0})} \sqrt{t_{j} }  
\end{aligned}
\end{equation}
and  
from the Gagliardo-Nirenberg interpolation inequality
we get 
\[
\| \gamma_{t_0} '  - \gamma_{t_{j}} '  \|_{L^\infty}
\leq C \| \gamma_{t_0} ''  - \gamma_{t_{j}} ''\|_{L^p}^\alpha \|\gamma_{t_0}   - \gamma_{t_{j}}  \|_{L^2}^{1-\alpha} 
\]
with $\alpha = \frac{3p}{5p-2}$. 
Since Lemma \ref{lem:EstSecDerivative} implies that the $L^p$-norm of the second derivatives of $\gamma_0$ and $\gamma_{t_j}$ are uniformly bounded,
 we conclude that
 \[
\| \gamma_{t_0} '  - \gamma_{t_{j}} '  \|_{L^\infty}
\leq C  (\sqrt{t_j} ) ^{1-\alpha} 
\]
for a constant $C>0$ depending on $p$, $\lambda$ and $E(\Gamma)$.
Furthermore, there exists a sufficiently small $T>0$ depending on $p$, $\lambda$ and $E(\Gamma)$ such that 
\begin{align}
\|\gamma_{t_j}'\|_{L^\infty}  
& \leq \|\gamma_{t_0}'\|_{L^\infty } + \|\gamma_{t_j} '-  \gamma_{t_0} '\|_{L^\infty} \nonumber \\
& \leq 1+ W +  C  (\sqrt{t_j} ) ^{1-\alpha} \nonumber \\
&< 2W
\end{align}
whenever $0< t_j<T$.
Since 
$$
 \|\gamma_{t_0} - \gamma_{t_j}\|_{L^\infty} \leq C \| \gamma_{t_0} ''  - \gamma_{t_{j}} ''\|_{L^p}^\beta \|\gamma_{t_0}   - \gamma_{t_{j}}  \|_{L^2}^{1-\beta} 
$$
with $\beta =\frac p {5p-2}$
by the Gagliardo-Nirenberg interpolation inequality,
we  also  have 
\begin{equation}
\|\gamma_{t_j}\|_{L^\infty} < 2 \mu
\end{equation}
whenever $0< t_j<T$.

In fact we can show that the same estimates hold for a suitably chosen minimising sequence.
Let us assume that $(\phi_n)$  is a minimising sequence for the functional $\mathcal{F}_j$ in the  class $\mathscr{V}$, i.e.
$
\mathcal{F}_j(\phi_n) \to \inf_{\phi \in \mathscr{V}} \mathcal{F}_j(\phi)
$
and note that $\mathcal{F}_j$ is bounded from below by construction. 
As  $\phi_{t_j}$ is still a competitor, we can assume without loss of generality that  
\[
\mathcal{F}_j( \phi_n) \leq \mathcal{F}_j(\phi_{t_j}) = E(\gamma_{t_j}) \leq E(\gamma_{t_0})
\]
for all $n \in \mathbb{N}$. In which case we can repeat the argument from the above  to obtain the bound  
\[
\| \gamma_{t_0} ' - \gamma_n'   \|_{L^\infty}  \leq C ( \sqrt{t_{j+1}})^{1-\alpha} 
\]
with $\gamma_n = \tld \gamma + \phi_n$.
It then follows that %
\begin{equation}
 \|\gamma_n'\|_{L^\infty} < 2 W
\end{equation}
for all $0<t_{j+1}< T$. 

\emph{Compactness}.
As a consequence of Lemma~\ref{lem:EstSecDerivative},
the minimising sequence 
$(\phi_n)$ is uniformly bounded in $W^{2,p}(\R/L\Z,\R^n)$.
It then follows that there exists a weakly converging subsequence
in $W^{2,p}(\R/L\Z,\R^n)$ which we also denoted by $(\phi_{n})$.
In addition, the Rellich-Kondra\v{s}ov compactness theorem implies that 
the subsequence $(\phi_{n})$ is strongly convergent in $C^1(\R/L\Z,\R^n)$. 
Let us denote the limit of this sequence by $\phi$. 
Since we have already establish that 
$\|\phi_n\|_{L^\infty} < 2\mu$
and $\|\phi_n'\|_{L^\infty} < 2W$,
it follows that 
$\|\phi\|_{L^\infty} < 2\mu$
and $\|\phi'\|_{L^\infty} < 2W$.
Therefore the limit $ \phi $ also belongs to $\mathscr{V}$. 

\emph{Lower semi-continuity}.
Let us finally prove that 
$$
 \mathcal{F}_j (\phi)   \leq \liminf_{n \rightarrow \infty} \mathcal{F}_j(\phi_n).
$$
As the $L^2$-term in the functional $\mathcal{F}_j$  
converges by the theorem of Rellich-Kondra\v{s}ov and the angle between $\tau$ and 
$\gamma_{t_j}'$ is uniformly bounded strictly away from $\tfrac \pi2$,
it suffices to show that 
\begin{equation}\label{eqn:Energylsc} 
 E^{(p)} (\tld \gamma + \phi)   \leq \liminf_{n \rightarrow \infty}  E^{(p)} (\tld \gamma + \phi_n).
\end{equation}
Note that the length term $\lambda \int_{\R/L\Z} ds$ appearing in the considered energy $E$, cf. \eqref{eqn:pEnergy}, can be dropped as well due to the convergence of the sequence $(\phi_n)$ in $C^1(\R/L\Z,\R^n)$. 
In order to prove \eqref{eqn:Energylsc} we use the curvature formula for $\kappa_{{\tld\gamma} + \phi_n}$
to  rewrite 
\[
E^{(p)} ( {\tld\gamma} + \phi_n)  
 = 
 \int_{\mathbb R /L \mathbb Z}  
 \frac { \left|P^\bot_{{\tld\gamma}' +  \phi_n'} ({\tld\gamma}'' + \phi_n '') \right|^p}
 {|{\tld\gamma}' + \phi_n'|^{2p}}   |{\tld\gamma}' + \phi_n'| ds 
\]
as the expression 
\begin{align*}
 E^{(p)}( {\tld\gamma}_n) &= \int_{\mathbb R / L \mathbb Z}  \frac { \left|P^\bot_{{\tld\gamma}' +  \phi'} ({\tld\gamma}'' + \phi_n '') \right|^p}{|{\tld\gamma}' + \phi'|^{2p}}   |{\tld\gamma}' + \phi'| ds 
 + \mathscr{I}_1 + \mathscr{I}_2 + \mathscr{I}_3
,\end{align*}
where 
\begin{align*}
 \mathscr{I}_1 & =  \int_{\mathbb R /L \mathbb Z}  
 \bigg(
 { \big  |P^\bot_{{\tld\gamma}' +  \phi_n'} ({\tld\gamma}'' + \phi_n '')  \big |^p}
 - 
 { \big  |P^\bot_{{\tld\gamma}' +  \phi'} ({\tld\gamma}'' + \phi_n '') \big|^p}
 \bigg)
 \frac{|{\tld\gamma}' + \phi_n'|}{|{\tld\gamma}' + \phi_n'|^{2p}}  ds
 \\
 \mathscr{I}_2 & =\int_{\mathbb R /L \mathbb Z}  
  \left(
 \frac 1 {|{\tld\gamma}' + \phi_n'|^{2p}} 
 - \frac 1 {|{\tld\gamma}' + \phi'|^{2p}}
 \right)
 \left|P^\bot_{{\tld\gamma}' +  \phi'} ({\tld\gamma}'' + \phi_n '') \right|^p 
  |{\tld\gamma}' + \phi_n'| ds  \\
 \mathscr{I}_3 & =\int_{\mathbb R / L \mathbb Z}  \frac { \left|P^\bot_{{\tld\gamma}' +  \phi'} ({\tld\gamma}'' + \phi_n '') \right|^p}{|{\tld\gamma}' + \phi_n'|^{2p}}   
 \bigg( |{\tld\gamma}' + \phi_n'| 
 -      |{\tld\gamma}' + \phi'|\bigg) ds 
.\end{align*}
The terms $\mathscr{I}_1$, $\mathscr{I}_2$ and  $\mathscr{I}_3$ vanish in the limit due to the convergence of the sequence $(\phi_n)$ in $C^1(\R/L\Z,\R^n)$ 
and the uniform bound on the $W^{2,p}$-norm of $\phi_n$. %
Moreover, the expression 
$$\mathscr{I}({\tld\gamma} + \phi_n) = \left( \int_{\mathbb R / L \mathbb Z}  \frac { \left|P^\bot_{{\tld\gamma}' +  \phi'} ({\tld\gamma}'' + \phi_n '') \right|^p}{|{\tld\gamma}' + \phi'|^{2p}}   |{\tld\gamma}' + \phi'| ds \right)^{\frac 1p}$$  
defines a norm  equivalent to that of the  $W^{2,p}$-norm. 
In which case \eqref{eqn:Energylsc} 
follows from the lower semicontinuity of norms under weak convergence.
\end{proof}

For later reference, let us also state the following a priori estimate for the piecewise linear interpolations 
that results 
from \eqref{eq:APriori_I} and  \eqref{eq:APriori_II}.

\begin{corollary} \label{cor:APrioriEstimate}
The piecewise linear interpolations  
 $$
   \phi^{(h)}(t, \,\cdot\,) 
   = \phi_{t_{j}} + \frac {t-t_j}{h}
   \big (\phi_{t_{j+1}} - \phi_{t_j} \big) , \quad t_j \leq t \leq t_{j+1},
 $$
 satisfies the estimates
 $$
  \| \phi_{t''}^{(h)} - \phi_{t'}^{(h)} \|_{L^2} \leq C \sqrt{t'' - t'}
 $$
 and
 $$
  \int_{t'}^{t''} \int_{\mathbb R /L \mathbb Z}
  |\partial_t \phi^{(h)}(t,s) |^2 ds dt 
  \leq C \big(E(\gamma_{t'}) - E(\gamma_{t''}) \big)
 $$
 for any $0\leq t'< t'' <   T<\infty$. 
\end{corollary}

\begin{remark}
We thus obtain a piecewise linearly interpolated  solution 
\begin{equation}\label{eqn:PLI}
\gamma_t^{(h)}  = \tld \gamma + \phi^{(h)}_t, \quad 0\leq t < T,
\end{equation} for  the minimising
movements scheme. 
\end{remark}

 \section{Weak solutions}
 
 \subsection{Euler-Lagrange equations for the approximations}
In order to improve the regularity of the approximations, 
we derive the Euler-Lagrange equations related to the minimising movement scheme.

We recall the following expression (cf.~\cite[Lemma~2.1]{Dziuk2002}) for the first variation of the $p$-elastic energy, namely  
 \begin{equation} \label{eq:FirstVariation}
 \delta_\psi E^{(p)} (\gamma) = \int_{\mathbb R /L \mathbb Z} |\kappa|^{p-2} \langle \kappa , \delta_\psi \kappa \rangle ds
 + \frac 1 p \int_{\mathbb R / L \mathbb Z} |\kappa|^p \langle \partial_s \gamma, \partial_s \psi\rangle ds
 \end{equation}
 where
 $
 \delta_\psi \kappa =  \left(\partial_s^2 \psi \right) ^\bot - \langle \kappa , \partial_s \psi \rangle \partial_s \gamma - 2\langle \partial_s \gamma ,   \partial_s \psi  \rangle \kappa
 $, cf.~Proposition~\ref{pro:FirstVariationPElastic}. The first variation of the length term appearing in the definition of the energy $E$, cf. \eqref{eqn:pEnergy}, is given by 
 \begin{equation} \label{eq:FirstVariationLengthTerm}
 	\delta_\psi \Big( \lambda \int_{\R/L\Z} ds\Big) = \lambda \int_{\R/L\Z} \langle \partial_s\gamma, \partial_s \psi\rangle ds.
 \end{equation}
 Combining \eqref{eq:FirstVariation} and \eqref{eq:FirstVariationLengthTerm} with the fact that 
 $$
 \partial_s \psi = \frac 1 {|\gamma'|} \partial_x \psi,
 $$
	where $|\gamma'|=|\partial_x \gamma|$,
 we get 
 \begin{align*}
 \partial^2 _s \psi 
 & = \frac 1 {|\gamma'|} \partial_x \Big( \frac 1 {|\gamma'|} \partial_x \psi \Big) 
  = \frac 1 {|\gamma'|^2} \partial_x ^2 \psi - \frac 1 {|\gamma'| ^3 } \Big\langle \frac {\gamma'}{|\gamma'|} ,  \gamma'' \Big\rangle \partial_x \psi  
 \end{align*}
 so that 
 $$
  \delta_\psi E (\gamma) = \int_{\mathbb R /L  \mathbb Z}  \frac {|\kappa|^{p-2} }{|\gamma'|} \langle \kappa, \partial_x^2 \psi \rangle  dx  + R(\psi)
 ,$$
 where $R(\psi)$ has the form
 $$
  R(\psi) = \int_{\mathbb R / L\mathbb Z} \langle b , \partial_x \psi  \rangle  dx
 $$
for some $b \in L^{\infty}L^{1}$, where as a notational shorthand  %
$L^{\infty}L^{1}$ stands for
$L^\infty\big([0,T), L^{1}(\mathbb R / \mathbb Z, \mathbb R^n)\big)$.

On the other hand, 
solutions of the minimising movement scheme solve 
 \begin{equation}
 \langle \partial_t \gamma , P_\tau^\bot \psi \rangle =  - \delta_\psi E (\gamma)
 \end{equation}
 for all $\psi \in \left( W^{2,p} \right) ^\bot _\tau$.
 Therefore we conclude that 
 \begin{equation}
  \int_{\mathbb R / L \mathbb Z}  \frac {|\kappa|^{p-2} }{|\gamma'|} \langle \kappa, \partial_x^2 \psi \rangle  dx  + \tld R(\psi) = 0
 ,\end{equation}
 where 
 $$
 \tld R(\psi) = \int_{\mathbb R / L \mathbb Z} \langle b, \partial_x \psi\rangle  dx 
 +\int_{\mathbb R / L \mathbb Z} \langle  P_\tau^\bot (\partial_t \gamma) , \psi \rangle  dx
 .$$

 \subsection{Higher regularity for the approximations}
  To deduce regularity from the equation above, 
  we consider a smooth local orthonormal basis
   $\nu_1,\ldots,\nu_{n-1}$ for our approximate normal spaces. 
   If $\psi$ is a test function that is decomposed into the form
 $$
  \psi = \sum_{i=1}^{n-1} \psi_i \nu_i
 $$
 such that the scalar functions  $\psi _i$ vanish away from the neighbourhood,
  we find that 
 $$
  \partial^2_x \psi  
  = 
  \sum_{i=1}^{n-1}  \Big( \partial_x^2 \psi _i \nu_i + 2 \partial_x \psi_i \partial_x \nu_i 
  + \psi_i 
\partial_x^2  \nu_i \Big).
 $$
Therefore the evolution equation for the approximation yields
 \begin{equation} \label{eq:evolequMM}
  \sum_{i=1}^{n-1} \int_{\mathbb R / L  \mathbb Z}  
  \frac {|\kappa|^{p-2} }{|\gamma'|}
  \partial_x^2 \psi_i
   \langle \kappa, P_\tau^\bot \nu_i \rangle   dx  = Q(h) ,
 \end{equation}
 where 
 $$
 Q(h) = 
  \int_{\mathbb R / L \mathbb Z} 
  \langle b_t ,  \partial_x \psi \rangle 
  +  \langle c_t , \psi \rangle +\langle  P_\tau^\bot (\partial_t \gamma(t, \,\cdot\,)) ,\psi \rangle dx.
 $$

The following lemma helps us to deduce regularity from this form of the equation.

\begin{lemma}[$L_{\text{loc}}^1$-estimates] \label{lem:L1locregularity}
Let $I = (a,b)$ be an open subset of $\R$.
If there exist 
functions $u$, $f$ and $F$ in 
$L_{\text{loc}}^1 (I)$ 
such that 
$$
 \int_I \left( u \partial_x^2 \varphi + F \partial_x \varphi \right) dx = \int_I f \varphi dx
$$
for all  $\varphi \in C^{\infty}_c (I)$, then 
$$
 u(x) =\int_a^x \left(F(y) + \int_a^y f(z) dz  \right) dy+m (x-a) + d
$$
with $d  = \lim_{x \searrow a} u(x)$ and 
$$
m(b-a) = \lim_{x  \nearrow b} u(x) - \left(\int_I \left(F(y) + \int_a^y f(z) dz  \right) dy + d\right)
.$$
Moreover, the function $u \in W^{1,1}(I)$ 
with  
$$
 \|u\|_{W^{1,1}} \leq C (\|f\|_{L^1} + \|F\|_{L^1}).
$$
\end{lemma}

\begin{proof}
Let us first set 
\begin{align*}
w(x)  &= F(x) +  \int_a^x f(y) dy \\
 v(x) &= \int_a^x w(y) dy
\end{align*}
and note that $v \in W^{1,1} (I)$ with $v'=w$.
Then integration by parts implies that 
\begin{align*}
 \int_I v(x) \partial_x ^2 \varphi (x) \, dx 
 & = - \int_{I} v'(x) \partial_x \varphi (x) \, dx \\
 &= -  \int_{I} \left(F(x) \partial_x \varphi(x) 
 + \Big( \int_a^x f(y) dy  \Big) \partial_x \varphi(x) \right)dx \\ 
 &= -\int_I \left(F(x) \partial_x \varphi(x) - f(x) \varphi(x) \right) dx
.\end{align*}
Therefore 
$$
 \int_I (u-v) \partial_x^2 \varphi \,dx =0
$$
for all $\varphi \in C^\infty_c (I)$. In which case $u-v$ is an affine function
from which %
the conclusion easily follows.
\end{proof}

We can now use the latter lemma to establish:

 \begin{theorem}[Higher regularity]
If   
$\gamma_t^{(h)}$ is a solution to the minimising movements scheme given by \eqref{eqn:PLI},
   there exists a constant $C>0$ independent of $h$ such that  
 	\begin{equation} 
 	  \left \| |\kappa|^{p-2} P^\bot_\tau \kappa \right\|_{L^2([0,T), W^{1,1}) } \leq C . 
 	\end{equation} 	
In particular,  we have 
$\kappa$ uniformly bounded in $L^2L^q$ 
and 
$\gamma'$ uniformly bounded in $W^{1,q}$  for all $1\leq q<\infty$. 
\end{theorem}

 \begin{proof}
 This higher regularity result 
 directly follows  from the application of Lemma~\ref{lem:L1locregularity}
   to our evolution equation for the 
 minimising movement scheme approximations. 
 In particular, from Corollary~\ref{cor:APrioriEstimate} we see that 
$\gamma_t^{(h)}$ satisfies
$$
 \int_{0}^T \int_{\mathbb R / L \mathbb Z} |\partial_t \gamma^{(h)} (t,s)|^2 ds dt \leq 
C E(\gamma_0).
$$
Applying Lemma \ref{lem:EstSecDerivative} to \eqref{eq:evolequMM} together with a covering argument hence yields
$$
 \Big\|\frac{|\kappa|^{p-2}}{|\gamma'|} P^{\bot}_\tau \kappa\Big\|_{L^2([0,T), W^{1,1}) } < C.
$$ 
Since $(\gamma^{(h)})'$ is uniformly bounded in $W^{1,1}$ and $W^{1,1}$ is a Banach algebra, this implies
$$
 \||\kappa|^{p-2}P^{\bot}_\tau \kappa\|_{L^2([0,T), W^{1,1}) } < C.
$$ 
 \end{proof}
 
 \subsection{Convergence to weak solutions}
 We will use the following result 
 in order to obtain the convergence of solutions. 
This result is crucial for the control of the terms involving the energy.

 \begin{theorem} \label{thm:Compactness}
 Let $\gamma_n = \gamma + \phi_n$ be a sequence bounded in $L^\infty W^{2,p} \cap C^{\frac 12}L^2$ such that $|\kappa_n|^{p-2}\kappa_n$  is uniformly bounded in $L^2W^{1,1}$. 
 Then
 there exists a subsequence $\gamma_{n_j}$ such that the curvatures $\kappa_{n_j}$ converge in $L^2W^{2,p}$. \end{theorem}
 
 The proof of this theorem relies on the following interpolation estimate.
 
 \begin{lemma} \label{lem:kappaconverge}
There exists a constant $C_0>0$ depending on $p$ such that for any  
$W^{2,p}$-curves $\gamma_1$ and $ \gamma_2$  
with curvatures $\kappa_1$ and $\kappa_2$ we have 
 \begin{equation*}
 \| \kappa_1 - \kappa_2\|_{L^p} \leq 
 C_0 (\||\kappa_1|^{p-2} \kappa_1\|_{L^2 W^{1,1}} + \||\kappa_2|^{p-2}\kappa_2\|_{L^2W^{1,1}})  \|\gamma'_1 - \gamma'_2\|_{L^2L^\infty}.
 \end{equation*}
 If these curves are furthermore approximate normal graphs over $\tld \gamma$ as for  the solutions to the minimising movement scheme, we get
  \begin{equation*}
 \| \kappa_1 - \kappa_2\|_{L^p} \\ \leq 
 C_0 (\||\kappa_1|^{p-2} P^\bot_\tau \kappa_1\|_{L^2 W^{1,1}} + \||\kappa_2|^{p-2} P^\bot_\tau \kappa_2\|_{L^2W^{1,1}})  \|\gamma'_1 - \gamma'_2\|_{L^2L^\infty}.
 \end{equation*}
 where now $C=C(\lambda, p, E(\Gamma)).$
 \end{lemma}
 
 \begin{proof} 
 First note that 
 $$
   \int |\kappa_1 - \kappa_2|^{p}  ds  \leq  C_0\int  \left( |\kappa_1|^{p-2} 
\kappa_1 - |\kappa_2|^{p-2} \kappa_2 \right) (\kappa_1 - \kappa_2) ds
 $$
 (cf.~\cite[\S 1, Lemma~4.4]{DiBenedetto1993}).
 Then integration by parts and H\"older's inequality imply that 
 \begin{align*}
  \int |\kappa_1 - \kappa_2|^{p}  ds &\leq - C_0\int \partial_s \left( |\kappa_1|^{p-2} \kappa_1 - |\kappa_2|^{p-2} \kappa_2 \right) (\partial_s \gamma_1 - \partial_s \gamma_2) ds
  \\ & \leq C_0 \left( \left\| |\kappa_1|^{p-2} \kappa_1 \right \|_{W^{1,1}}  + \left \| |\kappa_2|^{p-2} \kappa_2 \right \|_{W^{1,1}}  \right) \|\gamma_1' - \gamma_2 '\|_{L^{\infty}}.
 \end{align*}
 So by integrating over time and using H\"older's inequality again we get 
 \begin{align*}
 & \iint |\kappa_1 - \kappa_2|^{p}  ds dt \\
 &\qquad \leq C_0 \left(\left \| |\kappa_1|^{p-2} \kappa_1 \right \|_{L^2W^{1,1}}  + \left \|  |\kappa_2|^{p-2} \kappa_2 \right \|_{L^2W^{1,1}}  \right) \|\gamma_1' - \gamma_2 '\|_{L^2L^{\infty}}
. 
\end{align*}
For the second estimate we proceed in a similar way. We apply Lemma \ref{lem:EstSecDerivative} to improve the first inequality to 
$$
\int |\kappa_1 - \kappa_2|^{p}  ds  \leq  C_0\int  \left( |\kappa_1|^{p-2} 
P^\bot_\tau \kappa_1 - |\kappa_2|^{p-2} P^\bot_\tau \kappa_2 \right) (\kappa_1 - \kappa_2) ds.
$$
Integrating by parts then yields
\begin{multline*}
\int |\kappa_1 - \kappa_2|^{p}  ds \leq - C_0\int \partial_s \left( |\kappa_1|^{p-2} P^\bot_\tau \kappa_1 - |\kappa_2|^{p-2} P^\bot_\tau \kappa_2 \right) (\partial_s \gamma_1 - \partial_s \gamma_2) ds
  \\  \leq C_0 \left( \left\| |\kappa_1|^{p-2} P^\bot \kappa_1 \right \|_{W^{1,1}}  + \left \| |\kappa_2|^{p-2} P^\bot_{\tau}\kappa_2 \right \|_{W^{1,1}}  \right) \|\gamma_1' - \gamma_2 '\|_{L^{\infty}}
  . 
\end{multline*}
 \end{proof}

\begin{proof}[Proof of Theorem \protect{\ref{thm:Compactness}}]
  Using a diagonal argument and the compact embedding $W^{2,p}  \hookrightarrow L^2$, 
  we get  a subsequence  $\gamma_{n_j}$ converging in $L^2$ for all times $t \in \mathbb Q \cap [0,T)$ (and 
  hence for all $0\leq t <T$ due to the uniform bound in $C^{\frac 1 2} L^2$). 
 This result, together with the uniform bound on the $W^{2,p}$-Sobolev norm and interpolation estimates, implies that 
 $
  \gamma_{n_j} \rightarrow \gamma \in C^{\alpha} ([0,T), W^{1,\infty})
 $
 with
$ \alpha = \frac{p-1}{5p-2}$. 
 Thus $\gamma_{n_j}$ converge to $\gamma$ in $L^2W^{2,p}$
 by Lemma~\ref{lem:kappaconverge}. 
 \end{proof}

\begin{proof}[Proof of Theorem \protect{\ref{thm:weakSolution}}]
From the construction in 
Section~\ref{sec:Existsdiscretetime}
there exists a solution 
$\gamma^{(h)}_t$ to the minimising movement scheme given by \eqref{eqn:PLI}
for all $0\leq t <T$ up to some positive final  time $T$  that depends only on $p$, $\lambda$ and the energy
$E(\Gamma )$
of the initial data. 
We think of this solution as solving a discrete version of 
the negative $L^2$-gradient flow of $E$.
Theorem~\ref{thm:Compactness} and Corollary \ref{cor:APrioriEstimate} can then be applied to get a  subsequence that converges  
in $L^2W^{2,p}$ such that $\partial_t \gamma^{(h)}$ weakly converges in $L^2$.
Now in order to show that 
the limit satisfies the desired evolution equations, 
we use the fact that 
the solutions of the minimising movement scheme satisfy 
\begin{align} \label{eq:ELMM}
 \int _0^T \int_{\mathbb R / L \mathbb Z} 
 \langle  \partial_t ^\bot \gamma_t^{(h)} , \psi  \rangle ds dt 
 = \int_{\mathbb R / L \mathbb Z} \delta_{\psi_t} E (\gamma^{(h)}_t) dt 
\end{align}
for all test functions  $\psi \in C_c^\infty((0,T) \times \mathbb R / L \mathbb Z  ,\R^n)$. 

Let us now take a sequence  $h_n \to 0$
for which the solutions of the minimising movement scheme $  \gamma^{(h_n)}$ converge 
to a family of curves $\gamma$ in  $L^2W^{2,p}$ such that $\partial_t \gamma^{(h_n)}$ converges to $\partial_t \gamma$ weakly in $L^2([0,T), \mathbb R / L \mathbb Z)$. 
As $\gamma'^{(h_n)}$ converges strongly to  $\gamma'$ in $L^2$, we  see that the weak convergence of  $\partial_t \gamma^{(h_n)}$  to $\partial_t \gamma$  in $L^2$ implies
\begin{align} \label{eq:ConvergenceTerm2}
  \int _0^T \int_{\mathbb R / L\mathbb Z} 
  \langle \partial_t ^\bot \gamma^{(h_n)}_t ,   \psi  \rangle ds dt 
  \rightarrow  
  \int_0^T \int_{\mathbb R /L  \mathbb Z} 
  \langle \partial_t ^\bot  \gamma_t , \psi \rangle  ds dt.
\end{align}
Convergence for the right-hand side of \eqref{eq:ELMM} is also straight forward. 
If we denote by $\kappa_n$ the curvature of $\gamma^{(h_n)}_t$
and 
integrate
\eqref{eq:FirstVariation},  we find that 
\begin{align*}
 \int_0^T \delta_{\psi_t} E (\gamma^{(h_n)}_t ) dt  
& = \int_0^T \int_{\mathbb R / L \mathbb Z} |\kappa_n|^{p-2} \langle \kappa_n , \delta_\psi \kappa_n \rangle ds dt  \\
&\qquad   + \frac 1 p \int_0^T \int_{\mathbb R / L \mathbb Z} |\kappa_n|^p \langle \partial_s \gamma^{(h_n)}, \partial_s \psi \rangle ds dt \\
& \qquad + \lambda \int_0^T \int_{\mathbb R /L \mathbb Z} \langle \partial_s \gamma^{(h_n)}, \partial_s \psi \rangle ds dt . 
\end{align*}
Since $\kappa_n$ converges to $\kappa$ in $L^2L^p([0,T) \times \mathbb R / L \mathbb Z )$ and
$\p_s \gamma^{(h_n)}$ converges to $\p_s\gamma$ uniformly,
the second term on the right-hand side of the latter equation  
converges to the corresponding term for $\gamma$ in lieu of $\gamma^{(h_n)}$. 
One can deduce the same fact for the first term via the formula 
$$
 \delta_\psi \kappa_n =
 \left(\partial_s^2 \psi \right) ^\bot 
 - \langle  \kappa_n , \partial_s \psi \rangle \tau 
 -\langle  \partial_s \psi , \tau \rangle \kappa_n 
,$$
since it implies that $ \delta_\psi \kappa_n$ converges
to $ \delta_\psi \kappa$
 in $L^2 L^p([0,T) \times  \mathbb R /L  \mathbb Z ,\R^n  )$. %
Therefore we get 
\begin{align} \label{eq:ConvergenceTerm1}
\int_{0}^T \delta_{\psi_t} E (\gamma^{(h_n)}) dt \rightarrow \int_{0}^T
\delta_{\psi_t} E  (\gamma) dt
.\end{align}
In which case equations \eqref{eq:ELMM}, \eqref{eq:ConvergenceTerm2} and 
\eqref{eq:ConvergenceTerm1} imply that %
\[
 \int _0^T \int_{\mathbb R / L \mathbb Z} 
 \langle \gamma_t  ,\partial_t ^\bot   \psi\rangle ds dt 
 = - \int_0^T  \delta_{\psi_t} E (\gamma_t) dt.
\qedhere
\]
\end{proof}

\subsection{Flow in the direction of the normal velocity}
Using the fact that the unit tangent belongs to $W^{2,p}$
we can finally prove  Corollary~\ref{cor:weakSolution2} 
under the  conditions of Theorem~\ref{thm:weakSolution}.
 
\begin{proof}[Proof of Corollary~\ref{cor:weakSolution2}]
 In abuse of notation, let $\tau = \frac {\gamma'}{|\gamma'|} \in W^{2,p}(\mathbb R / L \mathbb Z, \mathbb R^n)$ be the unit tangent and the vectors $\nu_1, \ldots , \nu_{n-1}$ be a smooth local orthonormal basis of our approximate normal space. 
Due to the fact that  
any $\psi \in C_c^{\infty}(\mathbb R / L \mathbb Z, \mathbb R^n)$ can be written as 
 $$
  \psi = \psi_0 \tau + \sum_{i=1}^{n-1} \psi_i \nu_i  
 $$
with functions $\psi_i \in W^{2,p}(\mathbb R /L  \mathbb Z, \mathbb R^n)$,
we  find that  
\begin{align*}
\int_0^T \int_{\mathbb R / L \mathbb Z}  
 \langle \partial^\bot_t  \gamma ,\psi \rangle ds dt 
 & =  \sum_{i=1}^{n-1} \int_0^T \int_{\mathbb R /L  \mathbb Z} 
 \langle  \partial^\bot_t  \gamma, \psi_i \nu_i \rangle  ds dt  \\
 &= -\int_0^T \delta_{\psi_t} E(\gamma_t) dt
,\end{align*}
since  
both $\delta_{\psi_0 \tau} E (\gamma) =0$ and 
$\langle \partial_t^\bot \psi_0, \tau\rangle =0$. 
 \end{proof}

 \section{Epilogue}
\noindent
Although the minimising movement scheme leads in a rather straight forward way to the 
short-time existence of weak solution for our gradient flow, 
there are three key questions one would 
like to resolve, namely: 
\begin{enumerate} 
\item Are weak solutions unique and do they have long-time existence for $0\leq t<\infty$? 
\item  Can one  use test functions for the gradient flow  
that are not orthogonal to a quasi-tangent? 
\item Does our notion of 
solution depend on the choice of the reference curve and the approximate normal 
directions?
\end{enumerate}
For long-time existence it looks as if one could, in principle,
restart the flow and the above short-time 
existence result to get an eternal solution. However one should be aware that this solution 
might have kinks which our methods cannot rule out. 
If one has uniqueness and some way of modifying the approximate normal, 
long-time existence would be possible. 
Our Corollary~\ref{cor:weakSolution2}
 is a first indication that a more fastidious regularity theory is needed 
in order to resolve the above issues. 
 
 The question of uniqueness seems to be completely open. 
For the more standard 
non-homogeneous evolution equations involving the $p$-Laplace operator, papers 
discussing uniqueness have only appeared rather recently.
In particularly, the method used to prove  uniqueness in \cite{DallAcqua2018} breaks down for our curvature equations.

\appendix

\section{First variation for the $p$-elastic energy}
Recall that for closed curves  $\gamma\colon \mathbb R / \mathbb Z \rightarrow \mathbb R^n$ in the $W^{2,p}$-Sobolev class the $p$-elastic energy is given by
$$
E^{(p)}(\gamma) = \tfrac 1p \int_{\R / \Z} |\kappa|^p ds.
$$ 
For the convenience of the reader, we give further details on the derivation of its first variation. The upcoming statement is proven along the line of \cite[Lemma~2.1]{Dziuk2002}, for which we identify the arclength element by $ds = |\partial_x \gamma| \, dx$ and the arclength derivative by $\partial_s = |\partial_x \gamma|^{-1} \partial_x$. 

\begin{proposition}\label{pro:FirstVariationPElastic}
	The first variation of the $p$-elastic energy $E^{(p)}$ for $\gamma\in W^{2,p}(\R/\Z,\R^n)$ in direction of $\psi\in W^{2,p}(\R/\Z,\R^n)$ is given by  
	\begin{equation*} \label{eq:FirstVariationElastEn}
		\delta_\psi E^{(p)} (\gamma) = \int_{\mathbb R / \mathbb Z} |\kappa|^{p-2} \langle \kappa , \delta_\psi \kappa \rangle ds
		+ \tfrac 1 p \int_{\mathbb R / \mathbb Z} |\kappa|^p \langle \partial_s \gamma, \partial_s \psi\rangle ds.
	\end{equation*}
	where
	$
	\delta_\psi \kappa =  \left(\partial_s^2 \psi \right) ^\bot - \langle \kappa , \partial_s \psi \rangle \partial_s \gamma - 2\langle \partial_s \gamma,  \partial_s \psi \rangle \kappa
	$.
\end{proposition}
\begin{proof}
	We first observe
	\begin{align*}
		\delta_\psi E^{(p)} (\gamma) = \tfrac 1p \int_{\mathbb R / \mathbb Z} \delta_\psi (|\kappa|^{p} ) ds 
		+ \tfrac 1p \int_{\mathbb R / \mathbb Z} |\kappa|^p \delta_\psi( ds). 
	\end{align*}
	By applying the notation from above and the chain rule, we get
	\begin{align*}
		\delta_\psi (|\kappa|^{p} ) & =\tfrac{d}{d\ep}  \left[|\partial_s^2 (\gamma+\ep \psi)|^{p}  \right]_{\ep=0} \\
		& = \left[p \,  |\langle \partial_s^2(\gamma +\ep \psi), \partial_s^2(\gamma + \ep \psi)\rangle|^{\frac{p-2}{2}} \langle \tfrac{d}{d\ep} \partial_s^2(\gamma+\ep \psi), \partial_s^2(\gamma+\ep \psi) \rangle \right]_{\ep=0} \\
		& = p \, |\kappa|^{p-2} \langle \delta_\psi (\kappa), \kappa  \rangle 
	\end{align*}
	and
	\begin{align*}
		\delta_\psi (ds) & = \tfrac{d}{d\ep}\left[ |\partial_x(\gamma+\ep\psi)| dx\right]_{\ep=0} \\
		& = \left[ \tfrac{1}{|\partial_x(\gamma+\ep\psi)|} \langle \partial_x (\gamma + \ep \psi), \partial_x \psi \rangle dx \right]_{\ep=0} \\
		& =  \langle \tfrac{\partial_x \gamma}{|\partial_x \gamma|},\tfrac{\partial_x \psi}{|\partial_x \gamma|}\rangle  |\partial_x \gamma| dx \\
		& = \langle \partial_s \gamma, \partial_s \psi\rangle ds.
	\end{align*}
	Similarly, we achieve 
	\begin{align*}
		& \delta_\psi (\kappa)  =\tfrac{d}{d\ep}  \left[  \tfrac{1}{|\partial_x(\gamma + \ep \psi)|} \partial_x \left(\tfrac{1}{|\partial_x(\gamma+\ep\psi)|} \partial_x (\gamma + \ep \psi)\right)\right]_{\ep=0} \\
		& = \left[ - \tfrac{1}{|\partial_x(\gamma + \ep \psi)|^3} \langle \partial_x(\gamma + \ep\psi), \partial_x \psi \rangle \partial_x \left(\tfrac{1}{|\partial_x(\gamma+\ep\psi)|} \partial_x (\gamma + \ep \psi)\right)\right]_{\ep=0}  \\
		& \quad 
		+ \left[  \tfrac{1}{|\partial_x(\gamma + \ep \psi)|} \partial_x 
		(-\tfrac{1}{|\partial_x(\gamma+\ep\psi)|^3} \langle \partial_x(\gamma+\ep\psi),\partial_x h \rangle \partial_x (\gamma + \ep \psi)
		+ \tfrac{1}{|\partial_x(\gamma+\ep\psi)|} \partial_x h)
		\right]_{\ep=0} \\
		& = - \langle \tfrac{\partial_x \gamma}{|\partial_x\gamma|}, \tfrac{\partial_x \psi}{|\partial_x\gamma|} \rangle\tfrac{1}{|\partial_x\gamma|} \partial_x (\tfrac{\partial_x \gamma}{|\partial_x\gamma|}) 
		- \tfrac{1}{|\partial_x\gamma|} \partial_x \left( \langle \tfrac{\partial_x \gamma}{|\partial_x\gamma|}, \tfrac{\partial_x \psi}{|\partial_x\gamma|} \rangle \tfrac{\partial_x \gamma}{|\partial_x\gamma|} \right)
		+ \tfrac{1}{|\partial_x\gamma|} \partial_x(\tfrac{\partial_x \psi}{|\partial_x\gamma|}) \\
		& = - \langle \partial_s \gamma, \partial_s \psi\rangle \kappa - \partial_s(\langle \partial_s \gamma, \partial_s \psi\rangle \partial_s \gamma) + \partial_s^2 \psi
	\end{align*}
	and hence by rearranging and the Leibniz rule
	\begin{align*}
		\delta_\psi(\kappa) & = \partial_s^2 \psi - \langle \partial_s \gamma, \partial_s^2 \psi\rangle \partial_s \gamma - \langle \partial_s^2 \gamma, \partial_s \psi\rangle \partial_s \gamma - 2 \langle \partial_s \gamma, \partial_s\psi\rangle \partial_s^2 \gamma \\
		& = P^\perp_{\partial_s\gamma} (\partial_s^2 \psi) - \langle \partial_s \gamma, \partial_s \psi\rangle \kappa - 2 \langle \kappa, \partial_s \psi\rangle \partial_s \gamma.
	\end{align*}
\end{proof}

\bibliography{Master}

\begin{bibdiv}
\begin{biblist}

\bib{ambrosio1996}{article}{
      author={Ambrosio, Luigi},
      author={Soner, Halil~Mete},
       title={Level set approach to mean curvature flow in arbitrary
  codimension},
        date={1996},
     journal={J. Differential Geom.},
      volume={43},
      number={4},
       pages={693\ndash 737},
}

\bib{Baker}{unpublished}{
      author={Baker, Charles},
       title={The mean curvature flow of submanifolds of high codimension},
        date={2011},
        note={arXiv:1104.4409 [math.DG]},
}

\bib{DallAcqua2018}{article}{
      author={{Dall'Acqua}, Anna},
      author={{Laux}, Tim},
      author={{Lin}, Chun-Chi},
      author={{Pozzi}, Paola},
      author={{Spener}, Adrian},
       title={{The elastic flow of curves on the sphere}},
        date={2018},
     journal={{Geom. Flows}},
      volume={3},
       pages={1\ndash 13},
}

\bib{DallAcqua2017}{article}{
      author={{Dall'Acqua}, Anna},
      author={{Lin}, Chun-Chi},
      author={{Pozzi}, Paola},
       title={{A gradient flow for open elastic curves with fixed length and
  clamped ends}},
        date={2017},
        ISSN={0391-173X; 2036-2145/e},
     journal={{Ann. Sc. Norm. Super. Pisa, Cl. Sci. (5)}},
      volume={17},
      number={3},
       pages={1031\ndash 1066},
}

\bib{DallAcqua2019}{article}{
      author={{Dall'Acqua}, Anna},
      author={{Lin}, Chun-Chi},
      author={{Pozzi}, Paola},
       title={{Elastic flow of networks: long-time existence result}},
        date={2019},
        ISSN={2353-3382/e},
     journal={{Geom. Flows}},
      volume={4},
       pages={83\ndash 136},
}

\bib{DeGiorgiMM}{incollection}{
      author={De~Giorgi, Ennio},
       title={New problems on minimizing movements},
        date={1993},
   booktitle={{Boundary value problems for partial differential equations and
  applications. Dedicated to Enrico Magenes on the occasion of his 70th
  birthday}},
   publisher={Paris: Masson},
       pages={81\ndash 98},
        note={(repr. in \emph{Ennio De Giorgi: Selected papers}, Springer,
  2006, pp. 699--713)},
}

\bib{DeTurck1983}{article}{
      author={{DeTurck}, Dennis~M.},
       title={{Deforming metrics in the direction of their Ricci tensors}},
        date={1983},
     journal={{J. Differ. Geom.}},
      volume={18},
       pages={157\ndash 162},
}

\bib{DiBenedetto1993}{book}{
      author={DiBenedetto, Emmanuele},
       title={Degenerate parabolic equations},
      series={Universitext},
   publisher={Springer-Verlag},
        date={1993},
}

\bib{Dziuk2002}{article}{
      author={{Dziuk}, Gerhard},
      author={{Kuwert}, Ernst},
      author={{Schatzle}, Reiner},
       title={{Evolution of elastic curves in $\mathbb{R}^n$: Existence and
  computation}},
        date={2002},
     journal={{SIAM J. Math. Anal.}},
      volume={33},
      number={5},
       pages={1228\ndash 1245},
}

\bib{Escher1998}{article}{
      author={Escher, Joachim},
      author={Mayer, Uwe~F.},
      author={Simonett, Gieri},
       title={The surface diffusion flow for immersed hypersurfaces},
        date={1998},
     journal={SIAM Journal on Mathematical Analysis},
      volume={29},
      number={6},
       pages={1419\ndash 1433},
}

\bib{Foote}{article}{
      author={Foote, Robert~L.},
       title={Shorter notes: Regularity of the distance function},
        date={1984},
     journal={Proceedings of the American Mathematical Society},
      volume={92},
      number={1},
       pages={153\ndash 155},
}

\bib{hamilton1982NashMoser}{article}{
      author={Hamilton, Richard~S.},
       title={{The inverse function theorem of Nash and Moser}},
        date={1982},
     journal={Bull. Amer. Math. Soc. (N.S.)},
      volume={7},
      number={1},
       pages={65\ndash 222},
}

\bib{Hamilton1982a}{article}{
      author={Hamilton, Richard~S.},
       title={{Three-manifolds with positive Ricci curvature}},
        date={1982},
     journal={J. Differential Geom.},
      volume={17},
      number={2},
       pages={255\ndash 306},
}

\bib{Huisken1984}{article}{
      author={Huisken, Gerhard},
       title={{Flow by mean curvature of convex surfaces into spheres}},
        date={1984},
     journal={J. Differential Geom.},
      volume={20},
      number={1},
       pages={237\ndash 266},
}

\bib{Huisken1999}{incollection}{
      author={Huisken, Gerhard},
      author={Polden, Alexander},
       title={Geometric evolution equations for hypersurfaces},
        date={1999},
   booktitle={{Calculus of variations and geometric evolution problems
  (Cetraro, 1996), vol.~1713 of {\em Lecture Notes in Math.}}},
      series={Lecture Notes in Math.},
      volume={1713},
   publisher={Springer},
       pages={45\ndash 84},
}

\bib{Jakob2018}{article}{
      author={{Jakob}, Ruben},
       title={{Short-time existence of the M\"obius-invariant Willmore flow}},
        date={2018},
     journal={{J. Geom. Anal.}},
      volume={28},
      number={2},
       pages={1151\ndash 1181},
}

\bib{Kazdan1981}{article}{
      author={Kazdan, Jerry~L.},
       title={{Another proof of Bianchi's identity in Riemannian geometry}},
        date={1981},
     journal={{Proc. Amer. Math. Soc.}},
      volume={81},
       pages={341\ndash 342},
}

\bib{LeCrone2018}{article}{
      author={LeCrone, Jeremy},
      author={Shao, Yuanzhen},
      author={Simonett, Gieri},
       title={{The surface diffusion and the Willmore flow for uniformly
  regular hypersurfaces}},
        date={2020},
     journal={Discrete and Continuous Dynamical Systems. Series S},
      volume={13},
      number={12},
       pages={3503\ndash 3524},
}

\bib{Mayer2000}{article}{
      author={{Mayer}, Uwe~F.},
      author={{Simonett}, Gieri},
       title={{Self-intersections for the surface diffusion and the
  volume-preserving mean curvature flow}},
        date={2000},
        ISSN={0893-4983},
     journal={{Differ. Integral Equ.}},
      volume={13},
      number={7-9},
       pages={1189\ndash 1199},
}

\bib{Mueller2019}{article}{
      author={{M\"uller}, Marius},
      author={{Spener}, Adrian},
       title={{On the convergence of the elastic flow in the hyperbolic
  plane}},
        date={2020},
     journal={{Geom. Flows}},
      volume={5},
       pages={40\ndash 77},
}

\bib{Novaga2020}{article}{
      author={Novaga, Matteo},
      author={Pozzi, Paola},
       title={A second order gradient flow of $p$-elastic planar networks},
        date={2020},
     journal={SIAM Journal on Mathematical Analysis},
      volume={52},
      number={1},
       pages={682\ndash 708},
}

\bib{Okabe2018}{article}{
      author={{Okabe}, Shinya},
      author={{Pozzi}, Paola},
      author={{Wheeler}, Glen},
       title={A gradient flow for the $p$-elastic energy defined on closed
  planar curves},
        date={2020},
     journal={{Math. Ann.}},
      volume={378},
      number={1-2},
       pages={777\ndash 828},
}

\bib{Perelman2002}{unpublished}{
      author={{Perelman}, Grisha},
       title={{The entropy formula for the Ricci flow and its geometric
  applications}},
        date={2002},
        note={arXiv:math/0211159 [math.DG]},
}

\bib{Pozzetta2019}{article}{
      author={Pozzetta, Marco},
       title={A varifold perspective on the $p$-elastic energy of planar sets},
        date={2019-02-27},
      eprint={1902.10463},
}

\bib{Simonett2001}{article}{
      author={Simonett, Gieri},
       title={{The Willmore flow near spheres}},
        date={2001},
     journal={Differential Integral Equations},
      volume={14},
      number={8},
       pages={1005\ndash 1014},
}

\end{biblist}
\end{bibdiv}
  
\end{document}